\documentclass[a4paper,12pt]{amsart}
\usepackage{graphicx} \usepackage{amsmath}
\usepackage{latexsym}
\usepackage{amsthm}
\usepackage{autograph,epic,latexsym,bezier,caption,amsbsy,psfrag,color,enumerate,amsfonts,amsmath,amscd,amssymb}
\usepackage{epsfig}
\usepackage{graphics}
\usepackage[latin1]{inputenc}

\setloopdiam{17}\setprofcurve{10}
\newtheorem{defi}{Definition}[section]
\newtheorem{thm}[defi]{Theorem}
\newtheorem{lemma}[defi]{Lemma}
\newtheorem{prop}[defi]{Proposition}
\newtheorem{cor}[defi]{Corollary}

\newtheorem{example}[defi]{Example}

\newtheorem{remark}[defi]{Remark}

\begin{document}

\thanks{The first author was supported by Austrian Science Fund project FWF P24028-N18.}

\title{Metric compactification of infinite Sierpi\'{n}ski carpet graphs}

\author{Daniele D'Angeli}
\address{Institut f\"{u}r mathematische Strukturtheorie (Math C)\\
Technische Universit\"{a}t Graz \ \ Steyrergasse 30, 8010 Graz, Austria}
\email{dangeli@math.tugraz.at}
\author{Alfredo Donno}
\address{Universit\`{a} degli Studi Niccol\`{o} Cusano - Via Don Carlo Gnocchi, 3 00166 Roma, Italia \\
Tel.: +39 06 45678356, Fax: +39 06 45678379}
\email{alfredo.donno@unicusano.it}

\maketitle

\begin{abstract}
We associate, with every infinite word over a finite alphabet, an increasing sequence of rooted finite graphs, which provide a discrete approximation of the famous Sierpi\'{n}ski carpet fractal. Each of these sequences converges, in the Gromov-Hausdorff topology, to an infinite rooted graph. We give an explicit description of the metric compactification of each of these limit graphs. In particular, we are able to classify Busemann and non-Busemann points of the metric boundary. It turns out that, with respect to the uniform Bernoulli measure on the set of words indexing the graphs, for almost all the infinite graphs, the boundary consists of four Busemann points and countably many non-Busemann points.
\end{abstract}

\begin{center}
\end{center}

\section{Introduction}
The Sierpi\'{n}ski carpet fractal was introduced by
W. Sierpi\'{n}ski in 1916 \cite{courbe}, and it can be considered a generalization of the Cantor set in dimension $2$. Like the
well-known Sierpi\'{n}ski gasket, the carpet has a self-similar
structure: roughly speaking, this means that it is composed of $8$ smaller copies, with a scaling factor $3$,
that look exactly the same as it. The Sierpi\'{n}ski gasket is a  finitely ramified fractal (that is, it
can be disconnected by removing a finite number of points), whereas the carpet is an infinitely ramified fractal. Both the structures have been largely studied in the literature,
from different points of view. In particular, the study of critical
phenomena and physical models - the Ising model \cite{bonnier, ising, tutte2, gefen3, solvable}, the dimer model
\cite{dimeri}, the percolation model \cite{percolation} - on the Sierpi\'{n}ski carpet and on the
Sierpi\'{n}ski gasket has been the focus of several works in the
last decades.

In \cite{Rodi}, we have introduced an uncountable family of infinite rooted graphs, obtained as limit (in the Gromov-Hausdorff topology) of increasing sequences of rooted finite graphs, indexed by infinite words over a
finite alphabet. Such sequences represent a finite discrete approximation of the classical Sierpi\'{n}ski
carpet. For every word $w\in Y\times X^\infty$, with $Y=\{a,b,c,d\}$ and $X=\{0,1,2,3,4,5,6,7\}$, the infinite limit graph associated with $w$ is denoted $\Gamma_w$. We have studied in \cite{Rodi} the isomorphism properties of these
limit graphs, regarded as unrooted graphs, proving that there
exist uncountably many classes of isomorphism. The construction of the graphs $\{\Gamma_w\}_{w\in Y\times X^\infty}$ is recalled in Section \ref{sectionrodi}.\\ \indent
The aim of the present paper is to study the metric compactification of the graphs $\{\Gamma_w\}_{w\in Y\times X^\infty}$.
More precisely, our graphs are locally finite connected graphs, with a countable vertex set. This ensures that, when endowed with the standard geodesic distance, they are proper, complete, and locally compact metric spaces. These properties allow to apply to our graphs the theory of metric compactification developed in \cite{rieffel} and whose basic ideas are recalled in Section \ref{sectionrieffel}. For each of our graphs, the compactification is a space within which the vertex set, regarded as a metric space, embeds as an open and dense subspace. The points of the metric boundary are then defined as equivalence classes of horofunctions, where two horofunctions are equivalent if they differ by a constant. \\ \indent We will make large use of the characterization of the boundary points as limits of weakly-geodesic rays. A point of the boundary is said to be a Busemann point if it is the limit of almost-geodesic rays, which represent a special class of weakly-geodesic rays. In our computations, the base point of the horofunctions is represented by the root of the infinite graph. Moreover, the self-similar structure of the graphs allows us to give a complete description of the horofunctions, and so an explicit description of the metric boundary (Section \ref{sectioncompactification}). \\ \indent The study of horofunctions is a classical topic in the setting of $C^\ast$-algebras
and Cayley graphs of groups, in particular in connection with the investigation of the Cayley
compactification and the boundary of a group \cite{devin, rieffel}. We want to mention here the paper \cite{webster}, where the study of Busemann points is applied to the context of the metric boundary of Cayley graphs. See also the recent paper \cite{horodiestel}, where the authors study the
 horofunction boundary of the Lamplighter group, with respect to the word metric obtained from the generating set arising from viewing the
 Lamplighter group as a group generated by a finite automaton. Observe that the present paper follows the paper \cite{io}, where the same problems of isomorphism and horofunction classification are studied for a sequence of graphs approximating the Sierpi\'{n}ski gasket. On the other hand, the fact that the Sierpi\'{n}ski carpet is not finitely ramified makes much harder our study for the carpet graphs $\{\Gamma_w\}_{w\in Y\times X^\infty}$ than in the gasket case. It is worth mentioning that the isomorphism
problem has been studied also in \cite{JMD}, where the limits (in the Gromov-Hausdorff topology) of the sequences of rooted Schreier
graphs, associated with the action of the self-similar
Basilica group on the rooted binary tree, have been investigated.\\ \indent Our main result is given in Theorem \ref{teoremone}, where we provide an explicit description of the metric boundary of each graph $\Gamma_w$, for every $w\in Y\times X^\infty$; in particular, we are able to distinguish between Busemann and non-Busemann points. It also follows (see Corollary \ref{corosectioncomp}) that there exist uncountably many non-isomorphic graphs whose boundaries are isomorphic, and that the metric boundary always contains countably many non-Busemann points. Finally, we endow the set $Y\times X^\infty$ with the uniform Bernoulli measure, and we show (see Corollary \ref{measurethm}) that, with probability $1$, the boundary $\partial \Gamma_w$ consists of four Busemann points and countably many non-Busemann points.


\section{Metric compactification}\label{sectionrieffel}

In this section we recall the basic definition of metric compactification of a metric space $(X,d)$. We will mainly refer to the theory developed by Rieffel in
\cite{rieffel}.

Let $(X, d)$ be a metric space, and let $C_b(X)$ be the commutative algebra of continuous
bounded functions on $X$, with respect to the supremum norm. Let us fix a base point $z_0 \in X$. For each $y\in X$, the function $\varphi_y$ on $X$ is defined by
$$
\varphi_y(x) = d(x, z_0) - d(x, y), \qquad \forall x \in X.
$$
It is easy to check that $\varphi_y\in C_b(X)$. Let $H_d$ be the linear span in $C_b(X)$ of the set $\{\varphi_y: y\in X\}$: observe that $H_d$ does not depend on the particular choice of the base point. It can be easily seen that $\|\varphi_y-\varphi_z\|_\infty = d(y,z)$, so that the map $y\mapsto \varphi_y$ is an isometry from $(X, d)$ into $C_b(X)$. As the second space is complete, this isometry naturally extends to the completion of $X$.

If $(X,d)$ is a locally compact complete metric space, then it is possible to construct a compactification of $X$ within which $X$ is open and dense, to which each function $\varphi_y$ extends as a continuous function. This is the maximal ideal space $\overline{X}^d$ of the norm-closed subalgebra $\mathcal{G}(X,d)$ of $C_b(X)$ generated by the closed subalgebra $C_\infty(X)$ of functions vanishing at infinity, by the constant functions, and by $H_d$.

\begin{defi}
The space $\overline{X}^d$ is the metric compactification of $X$. The metric boundary of $X$ is the set $\overline{X}^d\setminus X$ and it is denoted by $\partial_d X$.
\end{defi}

The construction above is strictly related to the construction developed by Gromov in \cite{gromov}. Let $(X, d)$ be a geodesic locally compact complete metric space, and let $C(X)$ denote the vector space
of all continuous functions on $X$, endowed with the topology of uniform convergence on the compact subsets of $X$.
Let $C_\ast(X)$ denote the quotient of $C(X)$ modulo the subspace of constant
functions and, for every $f\in C(X)$, let $\overline{f}$ denote its image
in $C_\ast(X)$.

For each $y \in X$, we put $\psi_y(x) = d(x, y)$. This defines an embedding $\iota$ of $X$ into $C_\ast(X)$. Let $cl(X)$ be the closure of $\iota(X)$ in $C_\ast(X)$. Then it can be seen that $cl(X)$ is compact, and that $\iota(X)$ is open
in $cl(X)$, so that $cl(X)\setminus X$ is a boundary at infinity for $X$ (see, for instance, \cite[Chapter II.1]{quattro} or \cite[Part II.8]{sei}).

It is possible to show that there exists a homeomorphism between $\partial_d X$ and $cl(X)\setminus X$, defined by the map $u \mapsto \overline{g}_u$, with $g_u(x) = \lim_i (d(y_i,x)-d(y_i,z_0))$, where $z_0\in X$ is a fixed base point and $\{y_i\}_{i\in I}$ is a net of elements of $X$ converging to $u$.
\begin{defi}
For every $u\in \partial_d X$, the function $g_u$ is said the horofunction defined by $u$.
\end{defi}

The following definitions establish an explicit relationship between the metric boundary $\partial_d X$ and geodesic rays, or generalized geodesic rays, in $X$.

\begin{defi}\cite{rieffel}\label{defigeodesic}
Let $(X, d)$ be a metric space, and let $T$ be an unbounded
subset of $\mathbb{R}^+$ containing $0$. Let $\gamma: T \to X$. Then
\begin{enumerate}
\item $\gamma$ is a geodesic ray if $d(\gamma(t), \gamma(s)) = |t - s|$, for all $t, s \in T$;
\item $\gamma$ is an almost-geodesic ray if, for every $\varepsilon > 0$, there exists an integer $N$ such that, for every $t, s \in T$
with $t \geq s \geq N$, one has:
$$
|d(\gamma(t), \gamma(s)) + d(\gamma(s), \gamma(0)) - t| < \varepsilon;
$$
\item $\gamma$ is a weakly-geodesic ray if, for every $y \in X$ and every $\varepsilon > 0$,
there exists an integer $N$ such that, for every $s, t \geq N$, one has:
$$
|d(\gamma(t), \gamma(0)) - t| < \varepsilon \qquad \text{and} \qquad |d(\gamma(t), y) - d(\gamma(s), y) - (t - s)| < \varepsilon.
$$
\end{enumerate}
\end{defi}
Observe that any geodesic ray is an almost-geodesic ray; moreover, any almost-geodesic ray is a weakly-geodesic ray. The following crucial theorem holds.

\begin{thm}\cite{rieffel}
Let $(X, d)$ be a locally compact complete metric space,
and let $\gamma$ be a weakly-geodesic ray in $X$. Then the limit $\lim_{t\to +\infty} f(\gamma(t))$ exists
for every $f \in \mathcal{G}(X, d)$, and it defines an element of $\partial_d X$. Conversely, if $(X,d)$
is proper and if its topology has a countable base, then every
point of $\partial_d X$ is determined by a weakly-geodesic ray.
\end{thm}
The previous theorem leads to the following fundamental definition.
\begin{defi}
A point of $\partial_d X$ defined by an almost-geodesic
ray $\gamma$ is called a Busemann point of $\partial_d X$.
\end{defi}

It is quite interesting in general to establish if the metric boundary $\partial_d X$ of a metric space $(X,d)$ contains non-Busemann points. We will develop such investigation in Section \ref{sectioncompactification}, for an uncountable family of metric spaces given by the infinite Sierpi\'{n}ski carpet graphs defined in Section \ref{sectionrodi}. In fact, the construction described above can be applied to graphs $G=(V,E)$ satisfying some natural conditions, namely of being locally finite connected graph, with a countable vertex set. Recall that $V$ is a metric space when it is endowed with the standard geodesic distance $d$, where for every $x,y\in V$ the distance $d(x,y)$ is defined as the length of a minimal path from $x$ to $y$. Such a metric induces the discrete topology on $V$, so that every function on $V$ is continuous, and $(V,d)$ is a proper, complete and locally compact space. A fundamental property which holds in the graph setting is that, if $u\in \partial_d V$ is a Busemann point, that is, it is the limit of an almost-geodesic ray $\gamma$, then there exists a geodesic ray $\gamma'$ converging to $u$ \cite{webster}. As a consequence, this fact ensures that, in order to show that a point of $\partial_d V$ is not a Busemann point, it is sufficient to show that it is not the limit of any geodesic ray (this characterization will be used in the proof of the Proposition \ref{propnonbusemann}).


\section{Carpet graphs} \label{sectionrodi}
Let us start by fixing two finite alphabets $X=\{0,1,\ldots, 7\}$ and $Y=\{a,b,c,d\}$. For each $n\geq 1$, let $X^n=\{x_1x_2\ldots x_n : x_i\in X\}$ be the set of words of length $n$ over the alphabet $X$, and let $X^\infty = \{x_1x_2\ldots\ldots : x_i \in X\}$ be the set of infinite words over the alphabet $X$. The set $Y\times X^\infty$ can be equipped with the direct product topology. A basis of open sets is the collection of all cylindrical sets of type $C_{yu} = \{yuX^\infty : y\in Y, u \in X^n, \text{ for some }n\}$. The space $Y\times X^\infty$ is totally disconnected and homeomorphic to the Cantor set. The cylindrical sets generate a $\sigma$-algebra of Borel subsets of $Y\times X^\infty$. We will denote by $m$ the uniform Bernoulli measure on $Y\times X^\infty$.

Let $C_4$ denote the cyclic graph of length $4$, whose vertices will be denoted by $a, b, c, d$. We choose an embedding of this graph into the plane in such a way that $a$ is the left vertex of the bottom edge, and $b,c,d$ correspond to the other vertices by following the anticlockwise order (Fig. \ref{fig1}).\\
\indent Take an infinite word $w=yx_1x_2\ldots \in Y\times X^\infty$. We denote by
$w_n$ the prefix $yx_1\ldots x_{n-1}$ of length $n$ of $w$.

\begin{defi}
The infinite Sierpi\'{n}ski carpet graph $\Gamma_{w}$ is the rooted
graph inductively constructed as follows:
\begin{description}
  \item[Step $1$] The graph $\Gamma_w^1$ is the cyclic graph $C_4$ rooted at the vertex $y$.
  \item[Step $n\to n+1$] Take $8$ copies of $\Gamma_{w}^n$ and glue them
  together on the model graph $\overline{\Gamma}$, in such a way that these copies occupy the positions indexed by $0,1,\ldots, 7$ in $\overline{\Gamma}$ (Fig. \ref{fig1}). Note that each copy shares at most one (extremal) side with
  any other copy. As a root for the new rooted graph $\Gamma_w^{n+1}$, we choose the root of the copy of $\Gamma_w^n$ occupying the position indexed by the letter $x_n$. We identify the root of $\Gamma_w^{n+1}$ with the finite word $w_{n+1}=yx_1\ldots x_n$.
  \item[Limit] $\Gamma_{w}$ is the infinite rooted graph obtained as the limit of the sequence of finite rooted graphs
  $\{\Gamma_{w}^n\}_{n\geq 1}$, whose root is naturally identified with the infinite word $w$.
\end{description}
\end{defi}
The limit in the previous definition means that, for each $r>0$, there exists $n_0\in \mathbb{N}$ such that the ball $B_{\Gamma_w}(w,r)$ of radius $r$ rooted at $w$ in $\Gamma_w$ is isomorphic to the ball $B_{\Gamma_w^n}(w_n,r)$ of radius $r$ rooted at $w_n$ in $\Gamma_{w}^n$, for every $n\geq n_0$ (Gromov-Hausdorff topology).

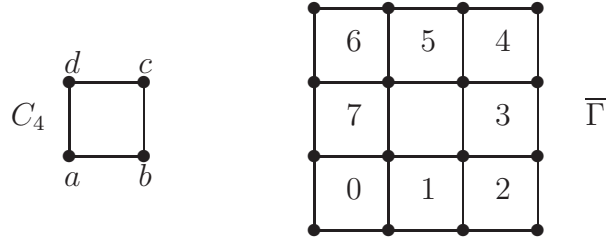
\begin{figure}
\begin{center}
\begin{picture}(250,100)\unitlength=0,14mm
\letvertex a=(30,70)\letvertex b=(100,70)\letvertex c=(100,140)\letvertex d=(30,140)

\letvertex e=(260,0)\letvertex f=(330,0)\letvertex g=(400,0) \letvertex h=(470,0)\letvertex i=(470,70)
\letvertex l=(470,140)\letvertex o=(330,210)\letvertex n=(400,210) \letvertex m=(470,210)
\letvertex p=(260,210)\letvertex q=(260,140)\letvertex r=(260,70)
\letvertex s=(330,70)\letvertex v=(330,140)\letvertex t=(400,70)\letvertex u=(400,140)
 \put(-25,100){$C_4$}   \put(515,100){$\overline{\Gamma}$}

\put(25,42){$a$} \put(25,149){$d$}\put(95,42){$b$}\put(95,149){$c$}

\drawvertex(a){$\bullet$}
\drawvertex(b){$\bullet$}\drawvertex(c){$\bullet$}\drawvertex(d){$\bullet$}

\drawvertex(e){$\bullet$}
\drawvertex(f){$\bullet$}\drawvertex(g){$\bullet$}\drawvertex(h){$\bullet$}\drawvertex(i){$\bullet$}
\drawvertex(l){$\bullet$}\drawvertex(m){$\bullet$}\drawvertex(n){$\bullet$}\drawvertex(o){$\bullet$}
\drawvertex(p){$\bullet$}\drawvertex(q){$\bullet$}\drawvertex(r){$\bullet$}\drawvertex(s){$\bullet$}
\drawvertex(t){$\bullet$}\drawvertex(u){$\bullet$}\drawvertex(v){$\bullet$}
\drawundirectededge(b,a){} \drawundirectededge(c,b){}
\drawundirectededge(a,d){} \drawundirectededge(c,d){}

\drawundirectededge(f,e){} \drawundirectededge(g,f){}
\drawundirectededge(h,g){} \drawundirectededge(i,h){}
\drawundirectededge(i,l){} \drawundirectededge(l,m){}
\drawundirectededge(n,o){} \drawundirectededge(m,n){}
\drawundirectededge(p,q){} \drawundirectededge(p,o){}
\drawundirectededge(q,r){} \drawundirectededge(r,e){}

\drawundirectededge(r,s){} \drawundirectededge(s,t){}\drawundirectededge(t,i){}
\drawundirectededge(f,s){}\drawundirectededge(s,v){} \drawundirectededge(v,o){}
\drawundirectededge(q,v){}\drawundirectededge(u,v){} \drawundirectededge(u,l){}
\drawundirectededge(g,t){}\drawundirectededge(u,t){} \drawundirectededge(u,n){}

\put(290,30){$0$}\put(290,100){$7$}\put(290,170){$6$}
\put(430,30){$2$}\put(430,100){$3$}\put(430,170){$4$}
\put(360,30){$1$}\put(360,170){$5$}
\end{picture}
\end{center}\caption{The cyclic graph $C_4$ and the model graph $\overline{\Gamma}$.}\label{fig1}
\end{figure}
Observe that, for all $v,w\in Y\times X^\infty$, the graph $\Gamma_v^n$ is isomorphic to $\Gamma_w^n$ as an unrooted graph. When we will refer to this unrooted graph, we will use the notation $\Gamma_n$. One can check that the number of vertices of $\Gamma_n$ is $\frac{11}{70}8^n+\frac{8}{15}3^n+\frac{8}{7}$, for every $n\geq 1$.

\begin{example}\rm
Let $w = a24\ldots \in Y\times X^\infty$. In Fig. \ref{figura} we have represented the finite graph $\Gamma^3_{w}$; its root is identified with the vertex $w_3=a24$.

\begin{figure}[h]
\begin{center}
\psfrag{a24}{$a24$}
\includegraphics[width=0.35\textwidth]{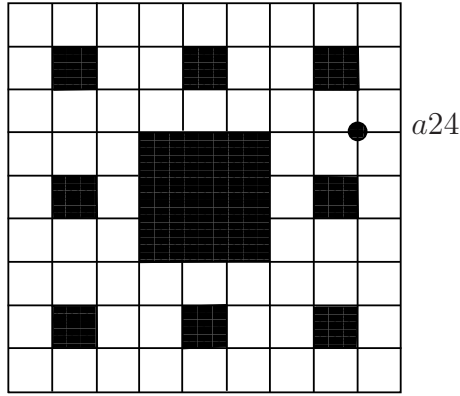}
\end{center}\caption{The rooted graph $\Gamma^3_{w}$, with $w=a24\ldots$.}\label{figura}
\end{figure}
 \end{example}

Note that, for each $n\geq 2$, the graph $\Gamma_n$ contains a central square, that we will call the \textit{hole} of level $n$, denoted by $H_n$, which does not contain any vertex of $\Gamma_n$ in its interior part. It is not difficult to check, by induction, that the number of the vertices of the boundary of $H_n$ is given by $4\cdot 3^{n-2}$, and that each side of the hole consists of exactly $3^{n-2}+1$ vertices. Moreover, due to the recursive construction of the graph $\Gamma_n$, one has that $\Gamma_n$ contains a hole of level $n$, as well as $8$ holes isomorphic to $H_{n-1}$, and more generally $8^k$ holes isomorphic to $H_{n-k}$, for every $ 2\leq k \leq n$. For each $n\geq 2$, we will denote by $s^n_1$ (resp. $s^n_3$, $s^n_5$, $s^n_7$) the top (resp. left, bottom, right) side of $H_n$. We also use the notation $A_n, B_n, C_n, D_n$ to denote the vertices of the hole $H_n$, ordered counterclockwise starting from the left vertex on the bottom side (see Fig. \ref{hole}).

\begin{figure}
\begin{center}
\psfrag{A}{$A_n$} \psfrag{B}{$B_n$} \psfrag{C}{$C_n$}     \psfrag{D}{$D_n$}
\psfrag{s1}{$s_1^n$} \psfrag{s3}{$s_3^n$}    \psfrag{s5}{$s_5^n$} \psfrag{s7}{$s_7^n$}
\includegraphics[width=0.22\textwidth]{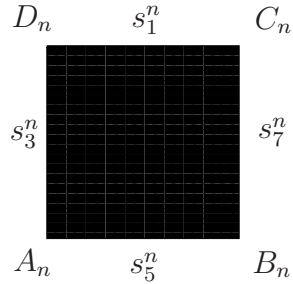}
\end{center} \caption{The hole $H_n$.}\label{hole}
\end{figure}

It is worth mentioning that two distinct finite words $v_n$ and $w_n$ may correspond to the same vertex of $\Gamma_n$, as shown in the following example.

\begin{example}  \rm
Let $u=a24\ldots$ and $w= d43\ldots$ be two infinite words in $Y\times X^\infty$. Then the graphs $\Gamma^3_{u}$ and $\Gamma^3_{w}$ are isomorphic as rooted graphs, even if the graphs $\Gamma^1_u$, $\Gamma^1_w$ and $\Gamma^2_{u}$, $\Gamma^2_{w}$ are not isomorphic as rooted graphs (see Fig. \ref{figurabis}).
\begin{figure}
\begin{center}
\psfrag{a}{$a$} \psfrag{a2}{$a2$} \psfrag{bis}{$a24\equiv d43$}
\psfrag{d}{$d$} \psfrag{d4}{$d4$}
\includegraphics[width=0.5\textwidth]{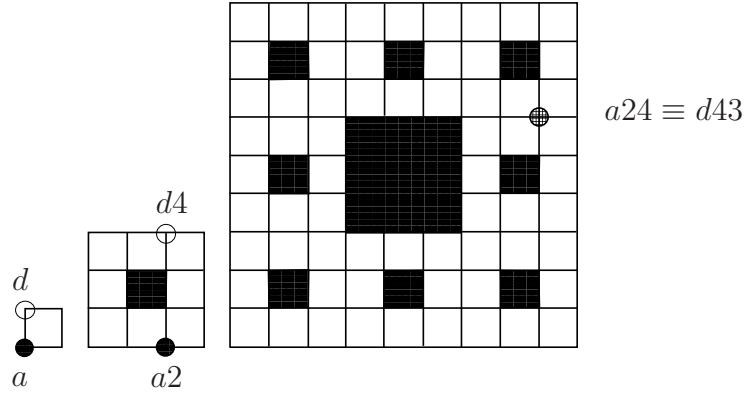}
\end{center} \caption{Construction of the rooted graphs $\Gamma^3_{u}$ and $\Gamma^3_{w}$, with $u=a24\ldots$ and $w=d43\ldots$.}\label{figurabis}
\end{figure}
\end{example}

We say that two infinite words $v,w\in Y\times X^\infty$ are \textit{cofinal} if they differ only for a finite number of letters. Cofinality is clearly an equivalence relation, that we will denote by $\sim$. Given $v,w\in Y\times X^\infty$, if there exists $n_0\in \mathbb{N}$ such that $v_n$ and $w_n$ correspond to the same vertex of the finite graph $\Gamma_n$ for every $n\geq n_0$, then it must be $v=v_{n_0}u$ and $w=w_{n_0}u$, for some $u\in X^{\infty}$; that is, $v$ and $w$ must be cofinal. On the other hand, it is not difficult to check that all the vertices belonging to the same infinite graph $\Gamma_w$, with $w\in Y\times X^\infty$, are cofinal with $w$. In \cite{Rodi} we gave the following results.

\begin{thm}\label{teoclass}
Let $v=yx_1x_2\ldots,w=y'x_1'x_2'\ldots\in Y\times X^\infty$ and let $G$ be the subgroup of $Sym(X)$ generated by the permutations $\{(1357),(04)(13)(57)\}$, isomorphic to the dihedral group of $8$ elements. Then the graphs
$\Gamma_v$ and $\Gamma_w$ are isomorphic, as unrooted graphs, if and only if there exists $\sigma\in G$ such that
$$
x_1'x_2'\ldots \sim \sigma(x_1x_2\ldots) :=\sigma(x_1) \sigma(x_2)\ldots.
$$
\end{thm}

\begin{cor}\label{corografi}
There exist uncountably many classes of isomorphism of graphs $\Gamma_w$, $w\in Y\times X^\infty$, regarded as unrooted graphs.
\end{cor}

\begin{example}\rm
In Fig. \ref{natale}, we have represented a finite part of the unrooted graph $\Gamma_{a0^\infty}$, where we have highlighted, by using black squares, the holes $H_n, n = 2, \ldots,5$ obtained in the first steps of the recursive construction of the infinite graph. We will use this kind of representation in the sequel of the paper.
\begin{figure}[h]
\begin{center}
\includegraphics[width=0.45\textwidth]{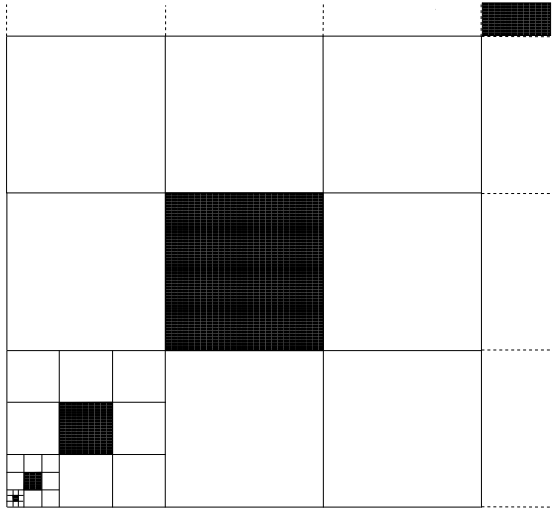}
\end{center}
\caption{A finite part of the unrooted graph $\Gamma_{a0^\infty}$.} \label{natale}
\end{figure}
\end{example}

\section{The metric compactification of the Sierpi\'{n}ski carpet graphs}\label{sectioncompactification}

Let $w = yx_1x_2\ldots \in Y\times X^\infty$, with $Y=\{a,b,c,d\}$ and $X=\{0,1,2,3,4,5,6,7\}$, and let us denote by $\Gamma_w$, as usual, the associated infinite carpet graph. In the sequel of this section, we will largely make use of an embedding of such a graph into the $2$-dimensional lattice $\mathbb{Z}^2$, which can be regarded as the Cayley graph of the Abelian group $\mathbb{Z}^2=\{\textbf{v} = (x,y): x,y\in \mathbb{Z}\}$ with respect to the symmetric generating set $\{\pm{\bf e}_1 = \pm (1,0), \pm {\bf e}_2=\pm(0,1)\}$. More precisely, this embedding is performed in such a way that the root $w$ of $\Gamma_w$ coincides with the point $(0,0)$ of the lattice, and each horizontal edge of $\Gamma_w$ coincides with an edge of $\mathbb{Z}^2$ connecting two vertices of type ${\bf v}$ and ${\bf v}\pm {\bf e}_1$, whereas each vertical edge of $\Gamma_w$ coincides with an edge of $\mathbb{Z}^2$ connecting two vertices of type ${\bf v}$ and ${\bf v}\pm {\bf e}_2$.

\subsection{Infinite growth and obstructions}\label{subsection1}

In this subsection, we describe what are the \lq\lq directions\rq\rq in which the graph $\Gamma_w$ can grow, and where the holes are settled with respect to the root $w$ in the recursive construction of the graph. We start with the following basic definitions, concerning the topological structure of an infinite carpet graph, that will be fundamental for the investigation and the classification of the metric boundary of the graphs $\{\Gamma_w\}_{w\in Y\times X^\infty}$. By using the embedding introduced above, we say that:
\begin{enumerate}
\item the graph $\Gamma_w$ has \textit{infinite growth} in the direction $d_1$ (resp. $d_3$, $d_5$, $d_7$) if its vertex set contains an (unbounded) subsequence of the vertex sequence $\{k{\bf e}_2\}_{k\in \mathbb{N}}$ (resp. $\{-k{\bf e}_1\}_{k\in \mathbb{N}}$, $\{-k{\bf e}_2\}_{k\in \mathbb{N}}$, $\{k{\bf e}_1\}_{k\in \mathbb{N}}$);
\item the graph $\Gamma_w$ has \textit{diagonal infinite growth} in the direction $d_{7,1}$ (resp. $d_{1,3}$, $d_{3,5}$, $d_{5,7}$) if its vertex set contains a sequence of vertices of type $\{a_k{\bf e}_1+ b_k{\bf e}_2\}$ (resp. $\{-a_k{\bf e}_1+ b_k{\bf e}_2\}$, $\{-a_k{\bf e}_1- b_k{\bf e}_2\}$, $\{a_k{\bf e}_1- b_k{\bf e}_2\}$), with
    $a_k,b_k \in \mathbb{N}$, and $a_k,b_k \to +\infty$ as $k\to +\infty$;
\item the graph $\Gamma_w$ has an \textit{obstruction} in the direction $d_1$ (resp. $d_3$, $d_5$, $d_7$) if there exists an infinite subset $M\subseteq \mathbb{N}$ and an increasing sequence $\{h_m\}_{m\in M} \subseteq \mathbb{N}$ such that the vertex $h_m{\bf e}_2$ (resp. $-h_m{\bf e}_1$, $-h_m{\bf e}_2$, $h_m{\bf e}_1$)
    belongs to the side $s^m_1$ (resp. $s^m_3$, $s^m_5$, $s^m_7$) of a hole isomorphic to $H_m$ in $\Gamma_w$, for every $m\in M$.
\end{enumerate}

\begin{remark}\label{13gennaio}\rm
It follows from the definition that, if the graph $\Gamma_w$ has an obstruction in the direction $d_i$, then it has infinite growth in the direction $d_i$, as well as diagonal infinite growth in the directions $d_{i-2,i}$ and $d_{i,i+2}$. Notice that the inverse implication is not true (for instance, the graph $\Gamma_{a0^\infty}$ in Fig. \ref{natale} has infinite growth in the directions $d_1$ and $d_7$, diagonal infinite growth in the direction $d_{7,1}$, but no obstruction).
\end{remark}

Now, for every $i=0,1,\ldots, 7$, we define:
$$
N_i = \left|\{j\in \mathbb{N} \ : \ x_j = i\}\right|.
$$
Keeping these definitions in our mind, we are able to prove the following proposition.

\begin{prop}\label{criteri_enne _i}
Let $w = yx_1x_2\ldots \in Y\times X^\infty$, with $X=\{0,1,2,3,4,5,6,7\}$ and $Y=\{a,b,c,d\}$.
\begin{enumerate}
\item Let $i=1,3,5,7$ and suppose $N_i = +\infty$. Then the graph $\Gamma_w$ has infinite growth in the three directions $d_j$, with $j\neq i+4 \mod 8$. Moreover, $\Gamma_w$ has an obstruction in the direction $d_i$.
\item Let $i=0,2,4,8$ and suppose $N_i = +\infty$. Then the graph $\Gamma_w$ has infinite growth in the two directions $d_{i+1}$ and $d_{i-1}$, where $i\pm 1$ must be taken modulo $8$.
\item Let $i=1,3,5,7$. Then the graph $\Gamma_w$ has diagonal infinite growth in the direction $d_{i,i+2}$ ($i+2$ must be taken modulo $8$) if at least one among $N_i, N_{i+1}, N_{i+2}$ is infinite.
    \end{enumerate}
\end{prop}
\begin{proof}
\begin{enumerate}
\item Suppose $N_7=+\infty$ (the other cases can be discussed similarly). Notice that if $x_m=7$, then, at the $(m+1)$-th step of the recursive construction of $\Gamma_w$, the root of the finite graph $\Gamma_w^{m}$ must be placed in the position $7$ of the model graph $\overline{\Gamma}$. This implies that in $\Gamma^{m+1}_w$ there is a hole isomorphic to $H_{m+1}$ on the right of the root $w$. As the number of letters of $w$ which are equal to $7$ is infinite, there must exist an unbounded set $M\subseteq \mathbb{N}$ and an increasing sequence $\{h_m\}_{m\in M}$ of natural numbers such that $h_m{\bf e}_1$ belongs to the side $s^m_7$ of a hole isomorphic to $H_m$ in $\Gamma_w$, for each $m\in M$, so that $\Gamma_w$ has an obstruction in the direction $d_7$.\\ \indent This also implies that $\Gamma_w$ has infinite growth in the direction $d_7$. On the other hand, in the recursive construction of the finite graphs $\Gamma_w^m$, the new biggest hole $H_m$ appears on the right of $w$, and this ensures the existence of an unbounded sequence of vertices of type $\{k{\bf e}_2\}_{k\in \mathbb{N}}$ in $\Gamma_w$, giving the infinite growth in the direction $d_1$. Analogously, the infinite growth in the direction $d_5$ can be proved.
\item  Suppose $N_0=+\infty$ (the other cases can be treated similarly). Notice that if $x_m=0$, then, at the $(m+1)$-th step of the recursive construction of $\Gamma_w$, the root of the finite graph $\Gamma_w^{m}$ must be placed in the position $0$ of $\overline{\Gamma}$. Since this happens for infinitely many indices, we are sure that there exist an unbounded sequence of vertices of type $\{k{\bf e}_2\}_{k\in \mathbb{N}}$ and an unbounded sequence of vertices of type $\{k{\bf e}_1\}_{k\in \mathbb{N}}$ in $\Gamma_w$. This implies that the graph $\Gamma_w$ has infinite growth in the directions $d_1$ and $d_7$, respectively.
\item Consider the case of the direction $d_{1,3}$ (the other cases can be treated similarly).  It follows from what we said above that if $N_1=+\infty$ (or $N_2=+\infty$, or $N_3=+\infty$) then the graph $\Gamma_w$ has infinite growth both in the direction $d_1$ and $d_3$, what easily implies the existence of a  sequence of vertices of type $\{-a_k{\bf e}_1+b_k{\bf e}_2 :  a_k,b_k\in \mathbb{N}; a_k,b_k \to +\infty \textrm{ as } k\to +\infty\}$. The claim follows.
    \end{enumerate}
\end{proof}

\begin{example}\rm
The graph $\Gamma_{a(67)^\infty}$ has infinite growth in the directions $d_1, d_5, d_7$; it has an obstruction in the direction $d_7$; finally, it has diagonal infinite growth in the directions $d_{7,1}$ and $d_{5,7}$ (see Fig. \ref{fig5}). On the other hand, the graph $\Gamma_{a(02)^\infty}$ has infinite growth in the directions $d_1, d_3, d_7$, but it has no obstruction; it has diagonal infinite growth in the directions $d_{1,3}$ and $d_{7,1}$ (see Fig. \ref{fig6}).

\begin{figure}
\begin{center}
\includegraphics[width=0.4\textwidth]{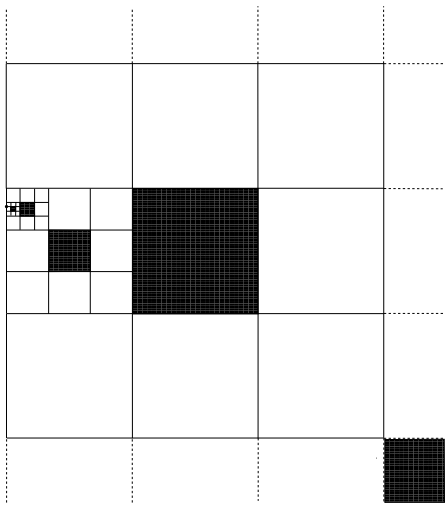}
\end{center}
\caption{A finite part of the rooted graph $\Gamma_{a(67)^\infty}$.} \label{fig5}
\end{figure}
\begin{figure}
\begin{center}
\includegraphics[width=0.4\textwidth]{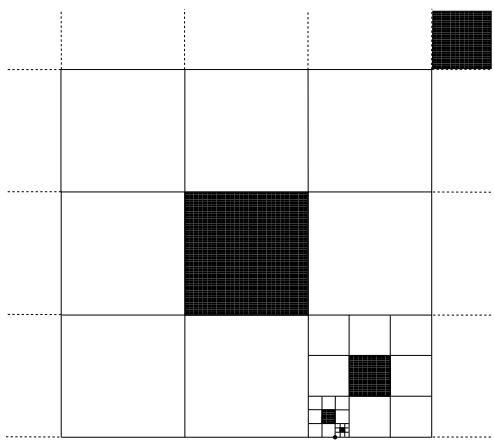}
\end{center}
\caption{A finite part of the rooted graph $\Gamma_{a(02)^\infty}$.} \label{fig6}
\end{figure}
\end{example}

\subsection{Classification of the metric boundary}\label{subsection2}
In the present subsection, we give the main results of the paper, consisting of the explicit classification of the metric boundary $\partial\Gamma_w$ of the graph $\Gamma_w$, for every $w\in Y\times X^\infty$.

We start by discussing the case of an infinite graph $\Gamma_w$ with obstruction. In order to simplify the notation we can suppose, without loss of generality, that the word $w$ is such that the graph $\Gamma_w$ has an obstruction in the direction $d_7$. From the definition of obstruction, we know that there exist an infinite subset $M\subseteq \mathbb{N}$ and an increasing sequence $\{h_m\}_{m\in M}\subseteq \mathbb{N}$ such that the vertex $h_m{\bf e}_1$ belongs to the side $s^m_7$ of a hole isomorphic to $H_m$ in $\Gamma_w$, for every $m\in M$. We will show in Lemma \ref{lemmauguaglianza} that it is possible to construct a sequence of vertices $\{z_m\}_{m\in M}$, each belonging to the side $s^m_7$ of the hole isomorphic to $H_m$, such that the length of a minimal path from $w$ to $z_m$ passing through the vertex $C_m$, is equal to the length of a minimal path from $w$ to $z_m$ passing through the vertex $B_m$. Observe that what we deduce in this specific situation $N_7=+\infty$ can be straightforwardly extended to the three other possible cases of obstruction.

In what follows, we denote by $\overline{d}(w,z_m)$ the length of a minimal path connecting $w$ and $z_m$ through the vertex $C_m$. Similarly, we denote by $\underline{d}(w,z_m)$ the length of a minimal path connecting $w$ and $z_m$ through the vertex $B_m$. Notice that, if $d$ is the geodesic distance in $\Gamma_w$, one has, by definition:
\begin{eqnarray}\label{novemberrevision}
\overline{d}(w,z_m)=d(w,C_m)+d(C_m,z_m) \ \textrm{ and } \  \underline{d}(w,z_m)=d(w,B_m)+d(B_m,z_m).
\end{eqnarray}

\begin{lemma}\label{lemmauguaglianza}
With the above notations, there exists a sequence $\{z_m\}_{m\in M}$, with $z_m \in s_7^m$ for every $m\in M$, such that:
$$
\overline{d}(w,z_m) = \underline{d}(w,z_m)= d(w, z_m).
$$
\end{lemma}
\begin{proof}
Notice that the term $d(w,C_m)$ in \eqref{novemberrevision} can be expressed as the sum of a horizontal and a vertical contribution. More precisely, if in the embedding of $\Gamma_w$ into $\mathbb{Z}^2$ the vertex $C_m$ has the representation $C_m(c_1(m), c_2(m))$, then $d(w,C_m)=c_1(m)+c_2(m)$ since, by construction, it must be $c_1(m), c_2(m) >0$, for $m$ large enough. In fact, the distance between $w$ and $C_m$ in the graph coincides with the distance between $w$ and $C_m$ regarded as vertices of $\mathbb{Z}^2$, by virtue of the embedding defined above. Now let $B_m(b_1(m), b_2(m))$. Note that it must be $b_1(m)>0, b_2(m) <0$, for $m$ large enough. We omit in what follows the dependence of the coordinates on $m$.

\indent Put $z:=c_1=b_1$ and observe that $c_2-b_2=3^{m-2}$, which is exactly the length of the side of $H_m$. Let $z_m$ be the vertex $(z,z_2)$. Note that it must be $z_2\in\{b_2,b_2+1, \ldots, b_2+3^{m-2}=c_2\}$, since the vertex $z_m$ belongs to the side $s^m_7$.  For our purpose, it is enough to choose $z_2$ such that $z_2-b_2=c_2$. In fact
\begin{eqnarray*}
\overline{d}(w,z_m)&=&d(w,C_m)+d(C_m,z_m) \\
&=& (c_1 +c_2)+(c_2-z_2)=b_1+z_2-b_2-b_2 \\
&=& b_1-b_2+(z_2-b_2)\\
&=& d(w,B_m)+d(B_m,z_m)=\underline{d}(w,z_m).
\end{eqnarray*}
\end{proof}
From now on, we denote by $\{z_m\}_{m\in M}$ the sequence obtained in the previous lemma, which will be said to be \textit{antipodal} to the root $w$ in the direction $d_7$.

\begin{remark}\rm
We stress once more the fact that the argument used in the proof of Lemma \ref{lemmauguaglianza} can be applied also to a graph $\Gamma_w$ with an obstruction in the direction $d_i$, with $i\neq 7$, proving that there exists an antipodal sequence to $w$ in the direction $d_i$ of the obstruction of $\Gamma_w$.
\end{remark}

In the following proposition, we show that the antipodal sequences of vertices in $\Gamma_w$ correspond to points of $\partial \Gamma_w$ which are not-Busemann. We will use again the embedding of $\Gamma_w$ in $\mathbb{Z}^2$.   We need the following lemma.

\begin{lemma} \label{lemmadicembre}
The sequence $\{z_m\}_{m\in M}$ forms a weakly-geodesic ray in $\Gamma_w$ and so it defines an element of $\partial \Gamma_w$.
\end{lemma}
\begin{proof}
Put $d_m = d(w, z_m)$, for each $m\in M$. Let $T=\{0\}\cup \{d_m : m\in M\}$. Define the map $\gamma: T \to \Gamma_w$ as $\gamma(0)=w$ and $\gamma(d_m) = z_m$, for each $m\in M$. Then this mapping defines a weakly-geodesic ray in $\Gamma_w$. In order to prove that, take an arbitrary vertex $y\in \Gamma_w$; put $z_m(z_1(m), z_2(m))$ and $y(y_1, y_2)$. Let us consider the case $y_2\geq 0$. By construction, we have $d(\gamma(d_m), \gamma(0))-d_m=0$ and
$$
d(\gamma(d_m),y) = d(z_m,y) = (c_2(m)-y_2)+(c_1(m)-y_1)+(c_2(m)-z_2(m)),
$$
and similarly for $d(\gamma(d_n),y)$. So it is straightforward to check that, for every $n\geq m$ large enough, one has:
$$
d(\gamma(d_n),y) - d(\gamma(d_m),y) - (d_n-d_m) = 0,
$$
since
$$
d_m = c_2(m) -z_2(m)+c_1(m)+c_2(m) \qquad \qquad    d_n = c_2(n) -z_2(n)+c_1(n)+c_2(n).
$$
The case $y_2<0$ is analogous, but in this case the geodesic connecting $y$ to $z_m$ and $z_n$ passes through $B_m$ and $B_n$, respectively. This concludes the proof, as the Property (3) of Definition \ref{defigeodesic} is satisfied.
\end{proof}

\begin{prop}\label{propnonbusemann}
The boundary point determined by the antipodal sequence $\{z_m\}_{m\in M}$ is not a Busemann point.
\end{prop}
\begin{proof}
First of all, notice that there exist at least two vertices $\textbf{v}_1, \textbf{v}_2$ in $\Gamma_w$ such that the shortest paths from $\textbf{v}_1$ to $z_m$ and from $\textbf{v}_2$ to $z_m$ do not intersect, for every $m$ (in the case of obstruction in the direction $d_7$ one can choose $\textbf{v}_1(0,1)$ and $\textbf{v}_2(0,-1)$, for example). Fix $\textbf{v}_1$ as the base point. Let $\gamma:\{0\}\cup M\to \Gamma_w$ be the weakly-geodesic ray such that $\gamma(m)=z_m$ and $\gamma(0)=\textbf{v}_1$.

If the boundary point determined by the antipodal sequence $\{z_m\}_{m\in M}$ is a Busemann point, then there exists a geodesic ray $\theta:\{0\}\cup M\to \Gamma_w$, with $\theta(0)=\textbf{v}_1$, such that
$$
\lim_{m\to +\infty} \varphi_v(\gamma(m))=\lim_{m\to +\infty}  \varphi_v(\theta(m)),
$$
for any $v\in \Gamma_w$.

Since the values taken by $\varphi_v$ are integer, then, for any $m$ sufficiently large, one has
$$
\lim_{m\to +\infty}  \varphi_v(\gamma(m))=\lim_{m\to +\infty}  \varphi_v(\theta(m))=\varphi_v(\gamma(m))=\varphi_v(\theta(m)):=f.
$$
From the definition of the function $\varphi_v$ and of weakly-geodesic and geodesic rays, we get, by choosing $v=\textbf{v}_2$ and $m$ large enough:
\begin{eqnarray}\label{formulaalfredo}
d(\gamma(m),\textbf{v}_2)=d(\theta(m),\textbf{v}_2)=m-f.
\end{eqnarray}
Moreover, for $m'<m$ large enough and with the choice $v=\theta(m')$, one gets:
$$
\varphi_{\theta(m')}(\gamma(m))= \varphi_{\theta(m')}(\theta(m)) = m-(m-m')=m'.
$$
This implies $d(\gamma(m), \theta(m'))= d(\gamma(m),\textbf{v}_1)- \varphi_{\theta(m')}(\gamma(m)) = m-m'$ and then
\begin{eqnarray*}
d(\gamma(m),\textbf{v}_1)&=& d(\theta(m'), \gamma(m))+ \varphi_{\theta(m')}(\gamma(m))\\
&=& d(\theta(m'), \gamma(m))+m'\\
&=& d(\theta(m'), \gamma(m))+ d(\theta(m'), \textbf{v}_1).
\end{eqnarray*}
On the other hand, we have by \eqref{formulaalfredo}
\begin{eqnarray*}
d(\gamma(m),\textbf{v}_2)&=&m-f= (m-m')+(m'-f)\\
&=& d( \gamma(m),\theta(m'))+ d(\theta(m'), \textbf{v}_2).
\end{eqnarray*}
This is impossible, because this would imply that there exist geodesic rays from $\textbf{v}_1$ to $z_m$ and from $\textbf{v}_2$ to $z_m$ both containing the common vertex $\theta(m')$. Absurd.
\end{proof}
The following propositions are given again for the special case of the obstruction in the direction $d_7$, but can be easily generalized.

\begin{prop}\label{propbc}
Let $\Gamma_w$ be an infinite carpet graph with an obstruction in the direction $d_7$. Then the sequences $\{B_m\}_{m\in M}$ and $\{C_m\}_{m\in M}$ define distinct Busemann boundary points.
\end{prop}
\begin{proof}
Because of the obstruction the vertices $B_m(b_1(m),b_2(m))$ and $C_m(c_1(m),c_2(m))$ are such that the sequences $b_1(m),b_2(m),c_1(m)$ and $c_2(m)$ are unbounded. It is not difficult to check that, for every vertex $v \in \Gamma_w$ and for all $m\geq n$ large enough, there exists a geodesic ray from $v$ to $B_n$ (resp. from $v$ to $C_n$) which can be extended to a geodesic ray from $v$ to $B_m$ (resp. from $v$ to $C_m$). This implies that the sequence $\{B_m\}_{m\in M}$ (resp. $\{C_m\}_{m\in M}$) gives rise to a Busemann boundary point. In order to show that these two boundary points are distinct, let us choose $v$ to be the vertex adjacent to $w$ with coordinates $(0,1)$, after the embedding into $\mathbb{Z}^2$. Then, it is straightforward to verify that $-1=\varphi_v(B_m)\neq \varphi_v(C_m)=1$, and this completes the proof.
\end{proof}

\begin{prop}\label{propequivalenza}
Let $w$ be such that $\Gamma_w$ has an obstruction in the direction $d_7$. Let $\{z_m(z_1(m), z_2(m))\}_{m\in M}$ be the sequence which is antipodal to $w$ in the direction $d_7$, and let $\{f_m(f_1(m), f_2(m))\}_{m\in M}$ (resp. $\{g_m(g_1(m), g_2(m))\}_{m\in M}$) be another sequence of vertices in $s_7^m$ such that
$$
\lim_{m\to+\infty} \left(f_2(m)-z_2(m)\right)=+\infty \qquad  (\text{resp.}\ \lim_{m\to+\infty} \left(z_2(m)-g_2(m)\right)=+\infty).
$$
Then the sequences $\{f_m\}_{m\in M}$ (resp. $\{g_m\}_{m\in M}$) and $\{C_m\}_{m\in M}$ (resp. $\{B_m\}_{m\in M}$) yield the same boundary point.
\end{prop}
\begin{proof}
Let $T=M$ and $\gamma(m)=f_m$, $\gamma'(m)=C_m$, for each $m\in M$. The claim follows if we prove that, for every $v\in \Gamma_w$, the values $ \varphi_v(\gamma(m))$ and $ \varphi_v(\gamma'(m))$ eventually coincide.  The proof for the sequences $\{g_m\}_{m\in M}$ and $\{B_m\}_{m\in M}$ is analogous and left to the reader.
Notice that, if eventually $\overline{d}(w,f_m)=d(w,f_m)$ holds, then the claim easily follows by observing that
\begin{eqnarray*}
\varphi_v(\gamma(m))&=& d(\gamma(m),w) - d(\gamma(m),v)\\
&=& (d(w, C_m)+d(C_m, f_m)) - (d(v, C_m)+d(C_m, f_m)) \\
&=& d(\gamma'(m),w) - d(\gamma'(m),v)\\
&=&\varphi_v(\gamma'(m)).
\end{eqnarray*}
By playing with the triangular inequalities, we have, for every vertex $v$:
\begin{eqnarray*}
\overline{d}(v,f_m)\leq d(v, w)+ \overline{d}(w,f_m),
\end{eqnarray*}
and
\begin{eqnarray*}
\underline{d}(v,f_m)\geq -d(v, w)+ \underline{d}(w,f_m).
\end{eqnarray*}
From this and the fact that $\underline{d}(w,z_m)=\overline{d}(w,z_m)$, we get
\begin{eqnarray*}
\underline{d}(v,f_m)-\overline{d}(v,f_m)&\geq& \underline{d}(w,f_m)- \overline{d}(w,f_m)-2d(v, w)\\
&=& \underline{d}(w,z_m)+ d(z_m,f_m) - \overline{d}(w,z_m)+ d(z_m,f_m)- 2d(v, w) \\
&=& 2 d(z_m,f_m)-2d(v, w) >0
\end{eqnarray*}
for $m$ large enough, since $d(z_m,f_m) = f_2(m)-z_2(m) \to +\infty$ as $m\to+\infty$. This means $\overline{d}(v,f_m)=d(v,f_m)$ and concludes the proof.
\end{proof}

By using a similar argument as in the proof of Lemma \ref{lemmadicembre}, it is possible to show that the sequences $\{f_m\}_{m\in M}$ and $\{g_m\}_{m\in M}$, described in the Proposition \ref{propequivalenza}, form weakly-geodesic rays in $\Gamma_w$.

\begin{cor}\label{propnonequiv}
Let $\{z_m\}_{m\in M}$, $\{f_m\}_{m\in M}$ and $\{g_m\}_{m\in M}$ be sequences of vertices of $\Gamma_w$ as in Proposition \ref{propequivalenza}. Then the corresponding boundary points are distinct.
\end{cor}
\begin{proof}
Since, by Proposition \ref{propequivalenza}, the sequences $\{f_m\}_{m\in M}$ and $\{g_m\}_{m\in M}$ are equivalent to the sequences $\{C_m\}_{m\in M}$ and $\{B_m\}_{m\in M}$, respectively, it is enough to show the claim for $\{z_m\}_{m\in M}$, $\{C_m\}_{m\in M}$ and $\{B_m\}_{m\in M}$. Let $z$, $f$ and $g$ be the corresponding boundary points, respectively. Then it follows from Proposition \ref{propbc} that $f\neq g$. Now let $\textbf{v}'=(0,-1)$. Then it is easy to check that $-1=\varphi_{\textbf{v}'}(C_m)\neq \varphi_{\textbf{v}'}(z_m)=1$, so that $z\neq f$. Analogously, one can prove that $z\neq g$ and we have proved the claim.
\end{proof}

\begin{remark}\rm
It follows from Proposition \ref{propbc} and Proposition \ref{propequivalenza} that the sequences $\{f_m\}_{m\in M}$ and $\{g_m\}_{m\in M}$ determine Busemann points of $\partial \Gamma_w$. On the other hand, we will prove in Proposition \ref{propbus} that, if the sequence $\{f_m\}_{m\in M}$ (resp. $\{g_m\}_{m\in M}$) is opportunely chosen, it is not an almost-geodesic sequence, even if the corresponding boundary point is Busemann.
\end{remark}

\begin{prop}\label{propbus}
Let $\{f_m\}_{m\in M}$ (resp. $\{g_m\}_{m\in M}$) as before, with the further property that $d(C_m, f_m)\to +\infty$ (resp. $d(B_m, g_m)\to +\infty$) as $m\to+\infty$. Then $\gamma:T\subseteq \mathbb{N}\to \Gamma_w$ such that $\gamma(t_i)=f_i$ defines no almost-geodesic, for every choice of the $t_i$'s.
\end{prop}
\begin{proof}
We give the proof only in the case of the sequence $\{f_m\}_{m\in M}$. The proof for the sequence $\{g_m\}_{m\in M}$ is analogous and left to the reader. For every $j\geq i$ large enough, we put $L:=|d(\gamma(t_i),\gamma(t_j))+d(\gamma(t_i), w)-t_j|$. Then, by performing explicit computations, one can check that:
\begin{eqnarray*}
L&=&|d(f_i,f_j)+d(f_i,w)-t_j| = | 2c_2(j) -f_2(j) + 2 c_2(i)- 2f_2(i) + f_1(j) -t_j|.
\end{eqnarray*}
Since $|2 c_2(i)- 2f_2(i)|\to+\infty$ as $i\to +\infty$, the value of $L$ cannot be uniformly bounded for every $i$. On the other hand, it should be $L<\varepsilon$ in order to have an almost-geodesic ray.
\end{proof}

By collecting the previous results, we deduce the following corollary for a graph $\Gamma_w$ with an obstruction in the direction $d_i$, for some $i\in\{1,3,5,7\}$.

\begin{cor}\label{infinitinonbus}
If $\Gamma_w$ has an obstruction in the direction $d_i$, then $\partial \Gamma_w$ contains countably many non-Busemann point $\{\zeta_k^{(i)}\}_{k\in \mathbb{Z}}$ given by the sequences at bounded distance from the corresponding antipodal sequence.
\end{cor}
\begin{proof}
Suppose that the obstruction is in direction $d_7$ (the other cases are analogous). The fact that the sequences at bounded distance from $\{z_m\}_{m\in M}$ define non-Busemann points follows from the proof of Proposition \ref{propnonbusemann}. Let $\{z_m^{(k)}\}_{m\in M}$ be the sequence obtained from $\{z_m\}_{m\in M}$ by shifting the $y$-coordinate by an integer number $k$: i.e., with the above notations, $z_1^{(k)}(m)=z_1(m)$ and $z_2^{(k)}(m)=z_2(m)+k$. The sequence $\{z_m^{(k)}\}_{m\in M}$ is not equivalent to the sequence
$\{B_m\}_{m\in M}$ and $\{C_m\}_{m\in M}$, as can be easily deduced by applying a similar argument as in the proof of Corollary \ref{propnonequiv}. Moreover, for any $k\neq 0$, $\{z_m\}_{m\in M}$ and $\{z_m^{(k)}\}_{m\in M}$ give rise to distinct boundary points. To prove that, it is enough to compare the values of $\varphi_{v_k}(z_m)=k$ and $\varphi_{v_k}(z_m^{(k)})=-k$, where $v_k=(0,-k)$ after the standard embedding into $\mathbb{Z}^2$. Similarly, one can prove that the sequences $\{z_m^{(k)}\}_{m\in M}$ and $\{z_m^{(h)}\}_{m\in M}$, with $h\neq k$, correspond to distinct boundary points.
\end{proof}

The following proposition describes Busemann points associated with diagonal infinite directions in $\Gamma_w$, and it can be considered a natural extension of Proposition \ref{propbc}. We give the assert and the proof in the particular case of diagonal infinite growth in the direction $d_{7,1}$. However, it can be easily generalized to the other diagonal infinite directions.

\begin{prop}\label{propdiagonale}
Let $\Gamma_w$ have diagonal infinite growth in the direction $d_{7,1}$. Then all the unbounded sequences consisting of vertices of type $\{a_k {\bf e}_1 + b_k{\bf e}_2 : a_k,b_k\in \mathbb{N}; a_k,b_k\to +\infty \textrm{ as } k \to +\infty\}$ give rise to the same Busemann boundary point in $\partial\Gamma_w$. In particular, if $\Gamma_w$ has an obstruction in the direction $d_7$, then $\partial\Gamma_w$ contains two Busemann points corresponding to any unbounded sequence of vertices of type $\{t_k {\bf e}_1 + s_k{\bf e}_2 : t_k,s_k\in \mathbb{N}; t_k,s_k\to +\infty \textrm{ as } k \to +\infty\}$ and $\{p_k {\bf e}_1 - q_k{\bf e}_2 : p_k,q_k\in \mathbb{N}; p_k,q_k\to +\infty \textrm{ as } k \to +\infty\}$.
\end{prop}
\begin{proof}
The proof follows by observing that a geodesic path from $w$ to $a_k\textbf{e}_1+b_k\textbf{e}_2$, for $k$ large enough, can be seen as a geodesic path in $\mathbb{Z}^2$ after the usual embedding of $\Gamma_w$. The uniqueness can be proven by considering vertices $a_k\textbf{e}_1+b_k\textbf{e}_2$ and $a_h\textbf{e}_1+b_h\textbf{e}_2$ with $k,h$ large enough, and by checking that, for every choice of $v(v_1,v_2)\in \Gamma_w$, one has:
$$
\varphi_v(a_k\textbf{e}_1+b_k\textbf{e}_2) = \varphi_v(a_h\textbf{e}_1+b_h\textbf{e}_2) = v_1+v_2.
$$
The second claim follows from the fact that, if $\Gamma_w$ has an obstruction in the direction $d_7$, then it has diagonal infinite growth in the directions $d_{5,7}$ and $d_{7,1}$, as we have already observed in Remark \ref{13gennaio}.
\end{proof}

For every diagonal direction $d_{i,i+2}$, we will refer to the Busemann points of Proposition \ref{propdiagonale} as the (unique) \textit{diagonal} Busemann points $\beta_{i,i+2}$ in the direction $d_{i,i+2}$, where the sum $i+2$ must be taken modulo $8$, as usual.

\begin{remark}\label{remarkdiagonale}     \rm
It is worth mentioning here that, if $\Gamma_w$ has an obstruction in direction $d_1$ and $d_7$, then all the sequences of vertices $\{v_n(x_n,y_n)\}_{n\in \mathbb{N}}$ in $\Gamma_w$ with $x_n, y_n\to +\infty$ give rise to the same point of $\partial \Gamma_w$. For example, if $M,M'$ are the subsets of $\mathbb{N}$ corresponding to the obstructions $d_1$ and $d_7$, respectively, then $\{C_m\}_{m\in M}$ and $\{C_{m'}\}_{m'\in M'}$ yield the same (Busemann) boundary point in $\partial \Gamma_w$. The same argument works for any pair of obstructions in the directions $d_i$ and $d_{i+2 \mod 8}$. The proof of this claim works as in Proposition \ref{propdiagonale}.
\end{remark}

In the next proposition, we investigate the case where the graph $\Gamma_w$ has infinite growth, without obstruction, in some direction $d_i$, with $i=1,3,5,7$. It turns out that, in this situation, there exist both Busemann and non-Busemann boundary points.

\begin{prop}\label{propinfbus}
Let $\Gamma_w$ have infinite growth but no obstruction in the direction $d_i$, for some $i\in\{1,3,5,7\}$. Then there exist countably many Busemann points $\{\xi_k^{(i)}\}_{k\in S'}$ and countably many non-Busemann points $\{\eta_k^{(i)}\}_{k\in S''}$ in $\partial \Gamma_w$, where $S'$ and $S''$ are left (or right, or bi)-infinite subsets of $\mathbb{Z}$.
\end{prop}
\begin{proof}
We can suppose, without loss of generality, that $i=7$. The other cases are analogous. Consider the usual embedding of $\Gamma_w$ into $\mathbb{Z}^2$. Since the graph $\Gamma_w$ is infinite in the direction $d_7$, but it has no obstruction in such direction, the size of the holes intersected by the sequence $\{\textbf{v}+m\textbf{e}_1\}_{m\in \mathbb{N}}$ is bounded, for every $\textbf{v}\in \mathbb{Z}^2$. Now let ${\bf v}_k(0,k)$ be a vertex of $\Gamma_w$. Observe that $k$ varies in a set $S$, where $S=\mathbb{Z}$ if $\Gamma_w$ has infinite growth also in the directions $d_1$ and $d_5$, whereas $S$ is a left-infinite (or right-infinite) subset of $\mathbb{Z}$ if $\Gamma_w$ has infinite growth only in one direction between $d_1$ and $d_5$. Notice that $S$ can be partitioned into two infinite subsets $S'$ and $S''$ such that the geodesic ray consisting of the vertices $\{\textbf{v}_k+m\textbf{e}_1\}_{m\in \mathbb{N}}$ is contained in $\Gamma_w$, for each $k\in S'$, and the sequence $\{\textbf{v}_k+m\textbf{e}_1\}_{m\in \mathbb{N}}$ is not eventually contained in $\Gamma_w$, for each $k\in S''$.

In the first case, the geodesic ray gives rise to Busemann boundary points. It is a standard argument to show that there exists $v\in \Gamma_w$ such that $\lim_m \varphi_v(\textbf{v}_k+m\textbf{e}_1)\neq \lim_m \varphi_v(\textbf{v}_{k'}+m\textbf{e}_1)$ for $k\neq k'$, so that we get all distinct boundary points.

In the second case, the sequence of vertices  intersects an infinite sequence of holes whose size is bounded. Let $s$ be the maximal size of a hole in such a sequence, and let $\{H_t\}_{t\in \mathbb{N}}$ be the subsequence consisting of all the holes of size $s$, so that any side of each $H_t$ has length $3^{s-2}$. Let $A_t(a_1(t), a_2(t))$,  $B_t(b_1(t), b_2(t))$, $C_t(c_1(t), c_2(t))$, $D_t(d_1(t), d_2(t))$ denote, as usual, the corner vertices of such holes, for every $t\in \mathbb{N}$, and define the vertex $z_t(b_1(t), b_2(t) + h)$, with $0\leq h\leq 3^{s-2}$.

Notice that in this case there exist two vertices $\textbf{u}(u_1,u_2)$ and $\textbf{u}'(u_1',u_2')$ such that the shortest paths from $\textbf{u}$ to $z_t$ and from $\textbf{u}'$ to $z_t$ do not intersect (to do this, it suffices to choose $u_2 > c_2(t)$ and $u_2'< b_2(t)$).

We can now proceed as in the proof of Proposition \ref{propnonbusemann} and deduce that the limit point of the sequence $\{z_t\}_{t\in \mathbb{N}}$ is not Busemann. Finally, it is a standard argument to show that the non-Busemann points obtained in this second case are all distinct.
\end{proof}

\begin{cor}\label{remarknoequiv}
Let $\Gamma_w$ be a graph with an obstruction in the direction $d_i$, and no obstruction in the direction $d_{i+2}$ (resp. $d_{i-2}$). Then the boundary points $\zeta_k^{(i)}$, $\xi_h^{(i+2)}$, $\eta_h^{(i+2)}$, and $\beta_{i,i+2}$ (resp. $\zeta_k^{(i)}$, $\xi_h^{(i-2)}$, $\eta_h^{(i-2)}$, and $\beta_{i-2,i}$) are distinct, for every choice of $k$ and $h$.
\end{cor}

Now we have all ingredients to state our main result, which is a classification theorem of the metric boundary of the graph $\Gamma_w$, for every $w\in Y\times X^\infty$.

\begin{thm}\label{teoremone}
For every $w \in Y\times X^\infty$, the boundary $\partial \Gamma_w$ consists of Busemann and non-Busemann points. More precisely, the following possibilities can occur.
\begin{enumerate}
  \item Suppose that there exists a constant $K>0$ such that $N_h\leq K$ for every $h\in \{1,3,5,7\}$, and let $\emptyset \neq I\subseteq \{0,2,4,6\}$ such that $N_i=+\infty$, for each $i\in I$. There are countably many Busemann points and countably many non-Busemann points for each of the infinite directions $d_j$, $j=i\pm 1$, $i\in I$. There is a unique Busemann point for every direction $d_{i-1,i+1}$.
      \item Let $\emptyset \neq I\subseteq \{1,3,5,7\}$ be the subset of indices such that $N_i=+\infty$, for every $i\in I$. Similarly, let $J\subseteq \{0,2,4,6\}$ be the subset of indices such that $N_j=+\infty$, for every $j\in J$. There are countably many non-Busemann points for every direction $d_i$, $i\in I$. There is a unique Busemann point for each of the directions $d_{i-2,i}$ and $d_{i,i+2}$. For any $i\in I$ and $j\in J$, there exist countably many Busemann points and countably many non-Busemann points for each of the infinite directions $d_{i-2}$, $d_{i+2}$, $d_{j-1}$ and $d_{j+1}$, provided that such indices $i-2, i+2, j-1, j+1$ are not in $I$ (the indices are considered without repetition).
\end{enumerate}
\end{thm}
\begin{proof}
(1) If $N_h$ is bounded for every $h\in\{1,3,5,7\}$, then the graph $\Gamma_w$ has no obstruction. In this case, we get countably many Busemann points $\xi_k^{(j)}$ and countably many non-Busemann points $\eta^{(j)}_k$, for $j=i\pm 1$ and $i\in I$, according with Proposition \ref{propinfbus}. Moreover, one has the diagonal Busemann point $\beta_{i-1,i+1}$ for every direction $d_{i-1,i+1}$, according with Proposition \ref{propdiagonale}.

(2) In this case, we have obstruction in the direction $d_i$, $i\in I$. This gives countably many non-Busemann points $\zeta_k^{(i)}$, according with Corollary \ref{infinitinonbus}. For every $i\in I$, one has two distinct diagonal Busemann points $\beta_{i-2,i}$ and $\beta_{i,i+2}$, according with Proposition \ref{propdiagonale}. Moreover, for every $i\in I$, one has countably many Busemann points $\xi_k^{(h)}$ and countably many non-Busemann points $\eta^{(h)}_k$, for $h\in \{i-2,i+2\}\setminus I$, according with Proposition \ref{propinfbus}. Finally, for every $j\in J$, one has countably many Busemann points $\xi_k^{(h)}$ and countably many non-Busemann points $\eta^{(h)}_k$,  for $h\in \{j-1,j+1\}\setminus I$, according with Proposition \ref{propinfbus}. Then the statement follows from Corollary \ref{remarknoequiv}.
\end{proof}

From Theorem \ref{teoclass}, Corollary \ref{corografi} and Theorem \ref{teoremone} we get the following corollary.

\begin{cor}\label{corosectioncomp}
\begin{enumerate}
\item There are uncountably many non-isomorphic graphs $\Gamma_w$ whose boundaries $\partial \Gamma_w$ are isomorphic.
\item For every $w\in Y\times X^\infty$, the boundary $\partial \Gamma_w$ contains countably many non-Busemann points.
\end{enumerate}
\end{cor}

We want to investigate now the boundary compactification of the graphs $\Gamma_w$ from the measure theoretic point of view. In other words, we want to answer the following question: \lq\lq How does $\partial \Gamma_w$ look like generically?\rq\rq Recall that the measure space in this setting is the space of infinite sequences in $Y\times X^{\infty}$, endowed with the uniform Bernoulli measure $m$. From Theorem \ref{teoremone}, we can deduce the following corollary.

\begin{cor}\label{measurethm}
Let $\partial \Gamma_w$ be the metric boundary of the graph $\Gamma_w$, $w\in Y\times X^{\infty}$. Then with probability $1$, with respect to the uniform Bernoulli measure $m$ on $Y\times X^{\infty}$, the boundary $\partial \Gamma_w$ consists of four Busemann points and countably many non-Busemann points.
\end{cor}
\begin{proof}
The boundary $\partial \Gamma_w$ consists of four Busemann points and countably many non-Busemann points if and only if $w\in Y\times  W$, where $W$ is the set of words in $X^\infty$ with the property that $N_i=+\infty$ for every $i\in X$ (Theorem \ref{teoremone}). A standard argument from elementary probability theory shows that $m(Y\times W)=1$. Then the claim follows.
\end{proof}

\begin{example} \rm
Consider the infinite word $w=b7^\infty \in Y\times X^\infty$, so that we have $N_7 =\infty$, and $N_i =0$ for each $i\neq 7$. Then, according with Proposition \ref{criteri_enne _i}, the infinite graph $\Gamma_w$ has:
\begin{itemize}
\item infinite growth in the directions $d_1, d_5$, and $d_7$;
\item an obstruction in the direction $d_7$;
\item infinite diagonal growth in the direction $d_{7,1}$ and $d_{5,7}$.
\end{itemize}

Let $M= \mathbb{N}$. Then, for each $m\in M$, there exists in $\Gamma_w$ a hole on the right of the root $w$, isomorphic to $H_{m+1}$, whose vertices have coordinates:
$$
A_m\left(3^{m-1}-1,-\frac{3^{m-1}-1}{2}\right) \qquad     B_m\left(2\cdot 3^{m-1}-1,-\frac{3^{m-1}-1}{2}\right)
$$
$$
C_m\left(2\cdot 3^{m-1}-1,\frac{3^{m-1}+1}{2}\right) \qquad D_m\left(3^{m-1}-1,\frac{3^{m-1}+1}{2}\right).
$$
The sequence $\{z_m\}_{m\in M}$ of Lemma \ref{lemmauguaglianza} consists of the points $z_m \left(2\cdot 3^{m-1}-1, 1\right)$, for all $m\in M$. Observe that the length of each side of the hole is $3^{m-1}$. In particular, we have $d(w,z_m) = 3^m-1$. As examples of sequences $\{f_m\}_{m\in M}$ and $\{g_m\}_{m\in M}$ (see Proposition \ref{propequivalenza}) we can choose the vertices
$$
f_m\left(2\cdot 3^{m-1}-1, \frac{3^{m-2}+1}{2}\right)    \qquad      g_m\left(2\cdot 3^{m-1}-1, \frac{1-3^{m-2}}{2}\right),
$$
which represent points of the side $s^{m+1}_7$ of the hole isomorphic to $H_{m+1}$, sited at two thirds or one third of the total length of the side $s^{m+1}_7$, respectively. With these choices, it is easy to check that one has:
$$
f_2(m) - z_2(m) = z_2(m) - g_2(m) = \frac{1+3^{m-2}}{2} \to +\infty.
$$
The boundary consists of: countably many non-Busemann points $\zeta_k^{(7)}$, $\eta_k^{(1)}$, and $\eta_k^{(5)}$; two diagonal Busemann points $\beta_{5,7}$ and $\beta_{7,1}$; countably many Busemann points $\xi_k^{(1)}$ and $\xi_k^{(5)}$.

\begin{figure}[h]
\begin{center}
\psfrag{z3}{$z_3$}

\psfrag{z4}{$z_4$} \psfrag{B3}{$B_3$}
\psfrag{C3}{$C_3$}
\psfrag{B4}{$B_4$}
\psfrag{C4}{$C_4$}
\includegraphics[width=0.5\textwidth]{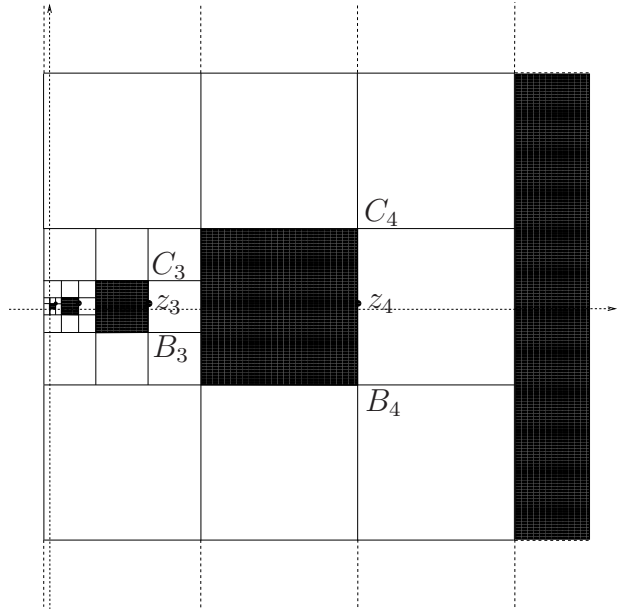} \end{center}  \caption{A finite part of the unrooted graph $\Gamma_{b7^\infty}$.}
\end{figure}
\end{example}

\section*{Acknowledgments}
Daniele D'Angeli was supported by Austrian Science Fund (FWF) P24028-N18.



\begin{thebibliography}{99}

\bibitem{quattro} W. Ballmann, Lectures on spaces of nonpositive curvature. With an appendix by Misha Brin. {\it DMV Seminar} {\bf 25}, Birkh\"{a}user Verlag, Basel, 1995.

\bibitem{bonnier} B. Bonnier, Y. Leroyer, C. Meyers, Critical exponents for Ising-like systems on
Sierpinski carpets, {\it J. Physique} {\bf 48} (Avril 1987),
553--558.

\bibitem{sei} M. R. Bridson, A. Haefliger, Metric spaces of non-positive curvature, \textit{Grundlehren der Mathematischen Wissenschaften [Fundamental Principles of Mathematical Sciences]} {\bf 319}, Springer-Verlag, Berlin, 1999.

\bibitem{io} D. D'Angeli, Horofunctions on Sierpi\'{n}ski type triangles, {\it Util. Math.}, in press.

\bibitem{Rodi} D. D'Angeli, A. Donno, Isomorphism classification of infinite Sierpinski carpet graphs, {\it AIP Conference Proceedings} {\bf 1648}, 570002 (2015).

\bibitem{JMD} D. D'Angeli, A. Donno, M. Matter, T. Nagnibeda, Schreier graphs of the
Basilica group, {\it J. Mod. Dyn.} {\bf 4} (2010), no. 1, 167--205.

\bibitem{ising} D. D'Angeli, A. Donno, T. Nagnibeda, Partition functions of the Ising model on some self-similar Schreier
graphs, in: {\it Progress in Probability: Random Walks, Boundaries
and Spectra} (D. Lenz, F. Sobieczky and W. Woess editors), {\bf
64} (2011), 277--304, Springer Basel.

\bibitem{dimeri} D. D'Angeli, A. Donno, T. Nagnibeda, Counting dimer coverings on self-similar Schreier
graphs, {\it European J. Combin.} {\bf 33} (2012), no. 7,
1484--1513.

\bibitem{devin} M. Develin, Cayley Compactifications of Abelian Groups, \textit{Ann. Comb.} {\bf 6}, No. 3-4 (2002), 295--312.

\bibitem{tutte2} A. Donno, D. Iacono, The Tutte polynomial of the Sierpi\'{n}ski and Hanoi graphs, {\it Adv. Geom.}, Vol. 13
(2013), Issue 4, 663--694.

\bibitem{gefen3} Y. Gefen, A. Aharony, B.B. Mandelbrot, Phase
transitions on fractals. III. Infinitely ramified lattices, {\it
J. Phys. A} {\bf 17} (1984), no. 6, 1277--1289.

\bibitem{gromov} M. Gromov, Hyperbolic manifolds, groups and actions. In {\it Riemann surfaces and related topics: Proceedings of the 1978 Stony Brook Conference}, 183--213. Princeton Univ. Press, Princeton, N. J., 1981.

\bibitem{horodiestel} K. Jones, G. A. Kelsey, The horofunction boundary of the Lamplighter group $L_2$ with the Diestel-Leader metric, available at \texttt{http://arxiv.org/abs/1410.8836}

\bibitem{rieffel} M. A. Rieffel, Group $C^\ast$-algebras as compact quantum metric spaces, {\it Doc. Math.} {\bf 7} (2002), 605--651.


\bibitem{percolation} M. Shinoda, Existence of phase transition of percolation on Sierpi\'{n}ski
carpet lattices, {\it J. Appl. Probab.} {\bf 39} (2002), no. 1,
1--10.

\bibitem{courbe} W. Sierpi\'{n}ski, Sur une courbe cantorienne qui contient une image
biunivoque et continue de toute courbe donnée, \textit{C. R. Acad.
Sci. Paris}, \textbf{162} (1916), 629--642.

\bibitem{webster} C. Webster, A. Winchester, Busemann points of infinite graphs, {\it Trans. Amer. Math. Soc.}, {\bf 358} (2006), no. 9, 4209--4224.

\bibitem{solvable} Z.R. Yang, Solvable Ising model on
Sierpi\'{n}ski carpets: The partition function, {\it Physical
review E} {\bf 49} (1994), no. 3, 2457--2460.


\end{thebibliography}
\end{document}